\documentclass[11pt]{article}

\usepackage{graphicx,amssymb,amsmath} \usepackage{epsfig} \usepackage{epsf}
\usepackage{amsfonts} \usepackage{latexsym}
\usepackage{amssymb} \usepackage{amsthm} \usepackage{graphics}
\usepackage{float} \usepackage{subfigure} \usepackage{graphicx,color} 
\usepackage{forloop}
\newtheorem{thm}{Theorem}[section]

\newtheorem{claim}[thm]{Claim}

\def\R{{\rm I\!R}}

\usepackage{tikz} \usetikzlibrary{decorations}
\usetikzlibrary{arrows} \usetikzlibrary{patterns,snakes}
\usetikzlibrary{calc,intersections,decorations.pathreplacing}
\usetikzlibrary{spy}

\definecolor{szurke}{gray}{0.9}

\title{Dense point sets with many halving lines}

\author{%
Istv\'an Kov\'acs\thanks{Budapest University of Technology and
  Economics. Supported by the National Research, Development and
Innovation Office, NKFIH, K-111827 and by the New National Excellence Program of the
Hungarian Ministry of Human Resources, UNKP-16-3-I.}
\and
G\'eza  T\'oth\thanks{Alfr\'ed R\'enyi Institute of Mathematics and Budapest University of Technology and
  Economics.
Supported by the National Research, Development and
Innovation Office, NKFIH, K-111827.  This work is connected to the scientific
program of the 
"Development of
quality-oriented and harmonized R+D+I strategy and functional model at
BME" project, supported by the New Hungary Development Plan (Project ID:
TÁMOP-4.2.1/B-09/1/KMR-2010-0002).  }}

\begin{document}
\maketitle

\centerline{\em Dedicated to the memory of Ricky Pollack.}

\begin{abstract}

A planar point set of $n$ points is called {\em $\gamma$-dense} 
if the ratio of the largest and smallest distances among the points 
is at most $\gamma\sqrt{n}$. 
We construct a dense set of $n$ points in the plane with 
$ne^{\Omega\left({\sqrt{\log n}}\right)}$ halving lines.
This improves the bound $\Omega(n\log n)$ of Edelsbrunner, Valtr and Welzl from
1997. 

Our construction can be generalized to higher dimensions, for any $d$ we 
construct a dense point set of $n$ points in $\R^d$ with 
$n^{d-1}e^{\Omega\left({\sqrt{\log n}}\right)}$ halving hyperplanes.
Our lower bounds are asymptotically the same as the best known lower bounds for
general point sets. 

\end{abstract}

\section{Introduction}

Let $P$ be a set of $n$ points in the plane in {\em general position}, 
that is, no three of them are on a line.
A line, determined by two points of $P$, is a {\em halving line} if it
has
exactly $(n-2)/2$ points of $P$ on both sides. 
Let $f(n)$ denote the maximum number of halving lines that a set of $n$ points
can have. It is a 
challenging unsolved problem 
to determine $f(n)$. 

The first bounds are due to 
Lov\'asz \cite{L71}, and Erd\H os, Lov\'asz, Simmons and Straus  
\cite{ELSS73}.  
They established the upper bound $O(n^{3/2})$, and the lower bound 
$\Omega\left(n\log n\right)$ (see also \cite{EW85} for a different lower bound
construction). 
Despite great interest in this problem,    
there was no progress until the very small improvement due to
Pach, Steiger and Szemer\'edi \cite{PSS92}. They improved the upper bound to
$O(n^{3/2}/\log^{*}n)$. 
The iterated logarithm of $n$, $\log^{*}n$, is the number of times the $\log$
function must be iteratively applied before the result is at most $1$.

The best known upper bound is  $O(n^{4/3})$,  
due to Dey \cite{D98}.
The lower bound has been improved by T\'oth \cite{T01} to
$ne^{\Omega\left({\sqrt{\log n}}\right)}$. 
Nivasch \cite{N08} simplified the construction of  T\'oth.

\smallskip

Suppose that $\gamma >0$. 
A planar point set $P$ of $n$ points is called {\em $\gamma$-dense} 
if the ratio of the largest and smallest distances determined by $P$ 
is at most $\gamma\sqrt{n}$. 
There exist arbitrarily large  $\gamma$-dense point sets if and only if 
$$\gamma\ge\sqrt{\frac{2\sqrt{3}}{\pi}},$$ see \cite{F43, EVW97}. 
Dense point sets are important in the analysis of
some geometric algorithms,
as they can be considered ``typical'' for some practical applications, like
computer graphics.

Recall that $f(n)$ is the maximum number of halving lines of a (not
necessarily dense) set of $n$
points in the plane. 
Edelsbrunner, Valtr and Welzl 
\cite{EVW97} 
showed, that a $\gamma$-dense point set can have at most
$O(\gamma\sqrt{n}f(3\gamma\sqrt{n}))$ halving lines. 
Agarwal and Sharir \cite{AS00} 
combined it with the result of Dey \cite{D98} and obtained that for any fixed
$\gamma$, a $\gamma$-dense point set can have at most
$O(n^{7/6})$ halving lines.

On the other hand, Edelsbrunner, Valtr and Welzl 
\cite{EVW97} 
constructed $2$-dense point sets 
with
$\Omega(n\log n)$ halving lines. 
Note that at that time $\Omega(n\log n)$ was the best known lower bound for
general point sets as well. 
In this note we give a better construction.

\medskip

\noindent {\bf Theorem 1.} {\em  For any even $n$, there exists 
a $4$-dense set of $n$ points in the plane with 
$ne^{\Omega\left({\sqrt{\log n}}\right)}$ halving lines.}


\medskip

Our lower bound is again asymptotically the same as the best known lower bound
for general point sets and it answers a question of Nivasch \cite{N08}.

\bigskip

Our construction can be generalized to higher dimensions. Let $d\ge 2$ and let 
$P$ be a set of $n$ points in the $d$-dimensional space, $\R^d$, 
in {\em general position}, that is, no $d+1$ of them are on a hyperplane.
A hyperplane, determined by $d$ points of $P$, is a {\em halving hyperplane} if it
has
exactly $(n-d)/2$ points of $P$ on both sides. 
Let $f_d(n)$ denote the maximum number of halving hyperplanes that a set of
$n$ points in $\R^d$ 
can have. So $f_2(n)$ is the same as $f(n)$ above. 
The best upper bounds are 
$f_3(n)\le O(n^{5/2})$ \cite{SST01}, $f_4(n)\le O(n^{4-1/18})$ \cite{S11}, and for $d>4$,
$f_d(n)\le O(n^{d-\varepsilon_d})$ where $\varepsilon_d=1/(2d)^d$ \cite{BMZ15}.
The planar lower bound construction can be ``lifted'' to higher dimensions and it gives the lower bound
$n^{d-1}e^{\Omega\left({\sqrt{\log n}}\right)}$ for $f_d(n)$ \cite{T01}. 

A set of $n$ points in $\R^d$ is called {\em $\gamma$-dense} 
if the ratio of the largest and smallest distances determined by $P$ 
is at most $\gamma\sqrt[d]{n}$. 
It was shown by Edelsbrunner, Valtr and Welzl \cite{EVW97} that for any $d\ge 3$,
a dense point set in $\R^d$
can have at most $O(n^{d-2/d})$ halving hyperplanes.
On the other hand, to the best of our knowledge, so far there is no nontrivial lower bound construction for dense point sets
in $d\ge 3$ dimensions.

\medskip

\noindent {\bf Theorem 2.} {\em  Let $d\ge 3$. There exists a $\gamma>0$ with the following property. 
For every $n$, such that $n+d$ is even,
there exists 
a $\gamma$-dense set of $n$ points in the $d$-dimensional space with 
$n^{d-1}e^{\Omega\left({\sqrt{\log n}}\right)}$ halving hyperplanes.}

\medskip

Just like in the planar case, 
our lower bounds are asymptotically the same as the best known lower bounds for
general point sets. 
The planar construction is divided into three steps, presented in the
next three sections. 
The first step is the bulk of the result, and it is based on the ideas of Nivasch
\cite{N08}. However, we have to be more careful to keep the distances under
control. 
In the last section we show how to generalize it to higher dimensions.


\section{First construction}

\medskip

\noindent {\bf Definition 1.} (a) For any point set $S$ in the plane, if points $p, q\in S$
determine a halving line, then the segment $pq$ is called a {\em halving segment}.
(b) A {\em geometric graph} is a graph drawn in the plane with straight line segments as edges.

\medskip 

First, we construct point sets $S$
with $n$ points and $ne^{\Omega\left({\sqrt{\log n}}\right)}$ halving lines,
such that the 
diameter of $S$ is $1$ and 
the smallest distance is 
$\Omega(n^{-8})$. 
Then, using several, slightly modified copies of $S$, we increase the smallest
distance to $\Omega(n^{-1/2})$ while the number of halving lines remains
asymptotically the same.

\medskip

\noindent {\bf Definition 2.} We say that a set of points forms an {\em
arithmetic progression}, if its points  are on a
horizontal line and consecutive points are at the same distance. 
Its {\em size} is the number of points it contains, its 
{\em step} is the distance between the consecutive points, its {\em width} is
the distance between the first and last points. 

\bigskip

\noindent {\bf Lemma 1.} {\em For every 
$i>0$, 
there exists a point set $S_i$ 
with the following properties.
 
\noindent (a) The number of points $|S_i|=n_{i}=\Omega(2^{i^2/2+i/2})$,

\noindent (b) the number of halving lines of $S_i$, 
$m_i=n_ie^{\Omega\left({\sqrt{\log n_i}}\right)}$,

\noindent (c) for any two points of $S_i$, the difference of their 
$x$-coordinates is at most $1$ and at least 
$\Omega(n_i^{-8})$. }

\bigskip

\noindent {\bf Proof of Lemma 1.}

For every fixed {\em order} 
$o\ge 0$, we construct  
the geometric graphs $G^o_{i}(S^o_{i}, H^o_{i})$
recursively in {\em index} $i$, for $o\ge i\ge 0$.
Each edge in $H^o_i$ will be a {\em halving edge} of $S^o_i$, that is, it will 
correspond to a halving line of $S^o_i$.
The vertex set $S^o_{i}$ will contain two types of
points, {\em plain} and {\em bold}. 
Together with $S^o_{i}$, we construct $H^o_{i}$, a
set of halving segments of $S^o_{i}$. 
Each halving segment in $H^o_{i}$ will be determined
by a plain and a bold point of $S^o_{i}$. 
Note that $H^o_{i}$ might not contain all
halving segments of $S^o_{i}$. 

It will be clear from the construction that for every $o, o'\ge i$, 
$G^o_{i}(S^o_{i}, H^o_{i})$ and 
$G^{o'}_{i}(S^{o'}_{i}, H^{o'}_{i})$
represent the same {\em abstract} graph.

\medskip

Now we sketch the construction, then we explain it precisely. 

\smallskip

Let $S^o_0$ be a set of just two points, $(1,1)$, which is plain, and 
$(0,0)$, which is bold, 
$H^o_0$ is the segment connecting them.
Suppose that we already have $S^o_{i-1}$ and $H^o_{i-1}$. 
Substitute each {\em plain} point of $S^o_{i-1}$ with an arithmetic progression 
of $a_{i}$ 
{\em plain} points, of a very small step $\varepsilon^o_{i}$, and 
substitute  each {\em bold} point with an arithmetic progression of $a_{i}+1$ 
{\em plain} points, of the same step, $\varepsilon^o_{i}$. 
Finally, put a {\em bold } point very close to the midpoint of each halving 
segment of $S^o_{i-1}$. Add all $2a_i+1$ edges from the bold point to the
plain points replacing the endpoints of the halving 
segment of $S^o_{i-1}$.
Finally,  the diagonal construction (i.e., $o=i$) satisfies the conditions of Lemma~1.

\bigskip

Now we describe  the construction precisely. 
Let $S^o_{0}$ be a set of two points, $(1,1)$, which is plain, and $(0,0)$, 
which
is bold. The set $H^o_{0}$ contains the segment determined by the two points.

For $i\ge 1$, let
$$a_{i}=2^i,\ \ 
\varepsilon^o_{i}=2^{-(4o+4)i},\ \ 
d^o_{i}=a_i\varepsilon^o_{i}=2^{-(4o+3)i}.$$

Suppose again that we already have $S^o_{i-1}$ and $H^o_{i-1}$. 
For every {\em plain} point $p=(x,y)\in S^o_{i-1}$, 
let $p^{(k)}=(x+k\varepsilon^o_{i}, y)$, and replace $p$ with  
the {\em plain} points $p^{(k)}\in S^o_{i}$, $0\le k\le a_i-1$.
We call these points the {\em children} 
of $p$, and $p$ is the {\em parent} of
the new points.

For every {\em bold} point $b=(x,y)\in S^o_{i-1}$, 
let $b^{(k)}=(x-k\varepsilon^o_{i}, y)$, and replace $b$ with 
the {\em plain} points $b^{(k)}\in S^o_{i}$, $0\le k\le a_i$.
Again, we call these points the {\em children} of $b$, and $b$ is the {\em parent} of
the new points.

For each halving segment 
$s=pb\in H^o_{i-1}$, where $p=(x_p, y_p)$ is plain and $b=(x_b, y_b)$ is bold, 
add a bold point 
$$q=\left(\frac{x_p+x_b}{2}-\frac{\varepsilon^o_{i}}{4}, \frac{y_p+y_b}{2}\right).$$
Add the segments $qp^{(k)}$, $0\le k\le a_i-1$, and 
$qb^{(k)}$, $0\le k\le a_i$, to $H^o_{i}$. See Figure~\ref{fig:recursion}.
We say that these new segments in $H^o_{i}$ 
are 
the {\em children} 
of $pb\in H^o_{i-1}$, 
and $pb$ is the {\em parent} of the new segments and $q$. 
		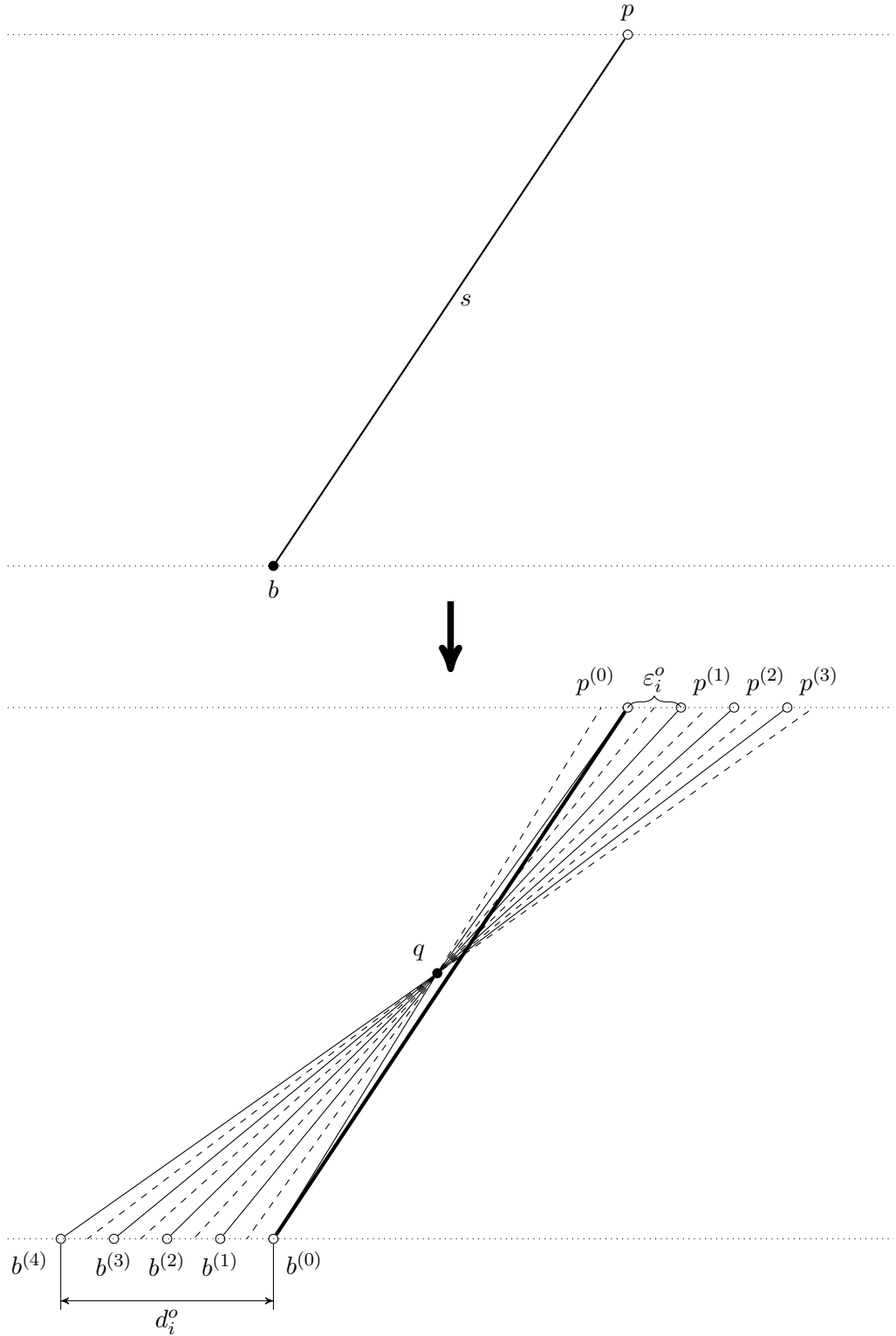
\begin{figure}
	\newcommand{\epsilonnull}{0.3}
		\begin{center}
		\begin{tikzpicture}[>=stealth', scale=2.7]

		 \tikzstyle{b}=[draw,circle,fill=black,minimum size=4,inner sep=0]
		 \tikzstyle{t}=[draw,circle,fill=white,minimum size=4,inner sep=0]
 
		 \node[b] (b) at (0,0) [label=below:$b$] {};
		 \node[t] (p) at (2,3) [label=above:$p$] {};

		 \draw[ thick] (b) -- (p) node[pos=0.5, right] {$s$};
		 \draw[ dotted] (-1.5,3) -- (3.5,3);
		 \draw[ dotted] (-1.5,0) -- (3.5,0);

	 \draw[->, line width=1mm] (1,-0.2) -- (1,-0.6);
	 	\begin{scope}[shift={(-6,-3.8)}]

	 \node[b] (bx) at (7-\epsilonnull/4,1.5) [label=above left:$q$] {};
	 \node[t] (p3) at (8 + 3*\epsilonnull,3)  [label=above right:$p^{(3)}$]  {};
	 \node[t] (p2) at (8 + 2*\epsilonnull,3) [label=above right:$p^{(2)}$]  {};
	 \node[t] (p1) at (8 + \epsilonnull,3) [label=above right:$p^{(1)}$] {};
	 \node[t] (p0) at (8,3) [label=above left:$p^{(0)}$] {};
	 \node[t] (b0) at (6,0)  [label=below right:$b^{(0)}$]  {};
	 \node[t] (b1) at (6-\epsilonnull,0)  [label=below:$b^{(1)}$]  {};
	 \node[t] (b2) at (6-2*\epsilonnull,0)  [label=below:$b^{(2)}$]  {};
	 \node[t] (b3) at (6-3*\epsilonnull,0) [label=below:$b^{(3)}$] {};
	 \node[t] (b4) at (6-4*\epsilonnull,0)  [label=below left:$b^{(4)}$]  {};
 
	 \draw[dashed] (bx) -- (8-\epsilonnull / 2, 3) node[pos=0.5, right] {};
	 \draw[ ] (bx) -- (p0) node[pos=0.5, right] {};
	 \draw[dashed] (bx) -- (8+\epsilonnull / 2, 3) node[pos=0.5, right] {};
	 \draw[ ] (bx) -- (p1) node[pos=0.5, right] {};
	 \draw[dashed] (bx) -- (8+3*\epsilonnull / 2, 3) node[pos=0.5, right] {};
	 \draw[ ] (bx) -- (p2) node[pos=0.5, right] {};
	 \draw[dashed] (bx) -- (8+5*\epsilonnull / 2, 3) node[pos=0.5, right] {};
	 \draw[ ] (bx) -- (p3) node[pos=0.5, right] {};
	 \draw[dashed] (bx) -- (8+7*\epsilonnull / 2, 3) node[pos=0.5, right] {};
	 \draw[ ] (bx) -- (b0) node[pos=0.5, right] {};
	 \draw[dashed] (bx) -- (6-\epsilonnull/2,0) node[pos=0.5, right] {};
	 \draw[ ] (bx) -- (b1) node[pos=0.5, right] {};
	 \draw[dashed] (bx) -- (6-3*\epsilonnull/2,0) node[pos=0.5, right] {};
	 \draw[ ] (bx) -- (b2) node[pos=0.5, right] {};
	 \draw[dashed] (bx) -- (6-5*\epsilonnull/2,0) node[pos=0.5, right] {};
	 \draw[ ] (bx) -- (b3) node[pos=0.5, right] {};
	 \draw[dashed] (bx) -- (6-7*\epsilonnull/2,0) node[pos=0.5, right] {};
	 \draw[ ] (bx) -- (b4) node[pos=0.5, right] {};
	 \draw[ dotted] (4.5,3) -- (9.5,3);
	 \draw[ dotted] (4.5,0) -- (9.5,0);
 
	  \draw[ultra thick] (b0) -- (p0) node[pos=0.5, right] {};

	\draw[ ] (b4) -- ++(0,-0.4) {};
	\draw[ ] (b0) -- ++(0,-0.4) {};
	\draw[style={<->}] (6-4*\epsilonnull,-0.35) -- (6,-0.35) node[pos=0.5, below] {$d_i^o$};
 
	  \draw [decorate,decoration={brace,amplitude=5pt}] (8, 3) -- (8 + \epsilonnull,3) node [pos=0.5, yshift=12pt] {$\varepsilon_i^o$};
		 \end{scope}
		\end{tikzpicture}
		\end{center}\caption{The recursion step on a halving segment $s$. 
	The new plain points $p^{(0)},\ldots, p^{(a_i-1)}$ are the children of $p$, and $b^{(0)},\ldots, b^{(a_i)}$ 
	are the children of $b$. The new bold point $q$ is the child of $b$, and $q$ is assigned to $s$.}\label{fig:recursion}
		\end{figure}

We also say that the point $q$ is the {\em child} of $b$ and $b$ is the {\em
  parent} of $q$.
Finally, we say that the point $q$ is 
{\em assigned to} segment $pb$.

Extend the relations ``parent'' and ``child'' to 
their transitive closure and call the resulting relation
{\em ``ancestor''} and {\em ``descendant''}, respectively.

A halving segment $s\in H^o_{i}$ is the 
{\em $k$-th
ancestor} 
({\em $k$-th descendant}) 
of $s'\in H^o_{i'}$, if it is an ancestor (descendant) and $i'=i+k$ ($i'=i-k$).
Similarly, a point $p\in S^o_{i}$ is the 
{\em $k$-th
ancestor} 
({\em $k$-th descendant}) 
of $p'\in S^o_{i'}$, if it is an ancestor (descendant) and $i'=i+k$ ($i'=i-k$).

For any $s\in H^o_{i}$ or $p\in S^o_{i}$
and $k>0$, let $Desc_k(s)$ or $Desc_k(p)$ denote the set of the $k$-th
descendants of $s$ or $p$, respectively, and $Desc(s)$ or $Desc(p)$ denote the set of all of its descendants.

\bigskip

Now we  prove the \emph{correctness of the construction}, that is, the segments in $H_i^o$ are halving segments of $S_i^o$.
\medskip

\noindent {\bf Definition 3.} 
(a) Let $\ell$ be a non-horizontal line and let $q$ be a point. The {\em horizontal
distance} of $\ell$ and $q$,  $d(\ell, q)$, is the length of the unique
horizontal segment, one of whose endpoints is $q$ and  the other one is on $\ell$.

(b) Let $s$ be a non-horizontal segment with endpoints $b =(x_1, y_1)$, 
$p = (x_2, y_2)$,
where $0\le y_1<y_2\le 1$, and let $\ell$ be its line. 
Let $\bar{b}=(\bar{x}_1, 0)$ and $\bar{p}=(\bar{x}_2, 1)$  
be the intersections of $\ell$ with the
line $(y=0)$ and $(y=1)$, respectively.
Call $\bar{s}=\bar{b}\bar{p}$ the {\em  extension} of $s$.
  
(c) For any segment $s$, the {\em horizontal strip} $HS(s)$ 
is the closed strip bounded by the horizontal lines through the endpoints of $s$. 

(d)  Let $r$ be a point in  $HS(s)$.
The {\em horizontal distance} of $s$ and $r$,  $d(s, r)$, is the 
horizontal distance of $r$ and the line determined by $s$.

(e) Let $\alpha>0$. The {\em $\alpha$-strip} of $s$, $S_{\alpha}(s)$, is the set of points $q$ such
that
$q\in HS(\bar{s})$ and $d(\bar{s}, q)\le\alpha$.
See Figure~\ref{fig:stripdef}.
\medskip

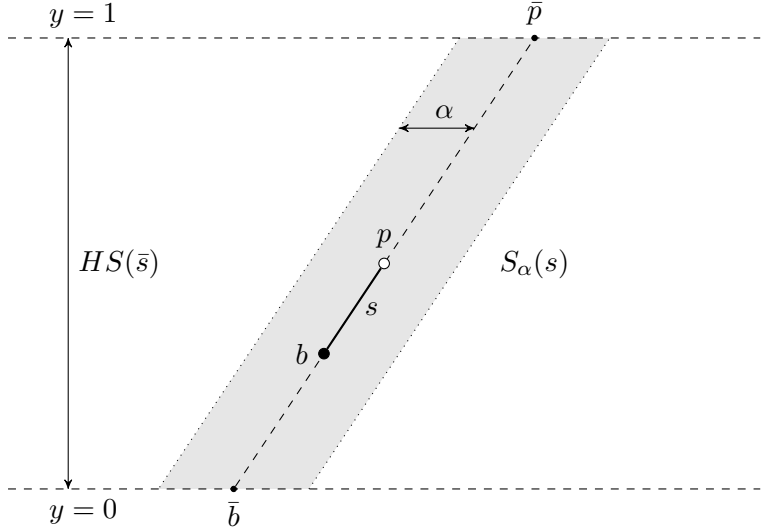
\begin{figure}
	
		\begin{center}
		\begin{tikzpicture}[>=stealth', scale=2]
			
	\fill [ fill=szurke] (-0.5,0) -- (0.5,0) -- (2.5,3) -- (1.5,3)  -- (-0.5,0);
	 \draw[ dotted] (0.5,0) -- (2.5,3);
 	\draw[ dotted] (1.5,3)  -- (-0.5,0);
		 \tikzstyle{b}=[draw,circle,fill=black,minimum size=4,inner sep=0]
		 \tikzstyle{t}=[draw,circle,fill=white,minimum size=4,inner sep=0]
		 \tikzstyle{tex}=[draw,circle,fill=black,minimum size=2,inner sep=0]
 
		 \node[tex] (b) at (0,0) [label=below:$\bar{b}$] {};
		 \node[tex] (p) at (2,3) [label=above:$\bar{p}$] {};
		  \draw[ dashed] (b) -- (p) node[pos=0.5, right] {};
 		\node[b] (px) at ($(p)!0.7! (b)$)  [label=left:$b$] {};
		\node[t] (bx) at ($(p)!0.5! (b)$)  [label=above:$p$] {};

		 \draw[ thick] (bx) -- (px) node[pos=0.5, right] {$s$};
		 \draw[ dashed] (-1.5,3) -- (3.5,3) node[pos=0.1, above] {$y = 1$};;
		 \draw[ dashed] (-1.5,0) -- (3.5,0) node[pos=0.1, below] {$y = 0$};;

	 \node at (2,1.5) {$S_\alpha (s)$};
		 \draw[<->] ($(b)!0.8!(p)$) -- ($(-0.5,0)!0.8! (1.5,3)$) node[pos=0.4, above] {$\alpha$};
		 \draw[<->] ($(-1.1,3)$) -- ($(-1.1,0)$) node[pos=0.5, right] {$HS(\bar{s})$};

		\end{tikzpicture}
		\end{center}
	\caption{The strips and the extension of $s$.}\label{fig:stripdef}
\end{figure}

We show that for any halving segment $s$  in $H^o_i$,
all of its descendant segments are in a very narrow strip of $s$, in a very strong
sense, 
while 
other points, that is, those points which are not endpoints of descendant
segments,  
will never get into that strip. See Figure~\ref{fig:strip2}.

\begin{claim}\label{stripproperty1}
Let $i\ge 0$ and $s\in H^o_i$. 
For every $r\in Desc(s)$, 
the extension $\bar{r}$
of $r$ lies in 
 $S_{\alpha}(s)$, where $\alpha=2^{i+2}d^o_{i+1}$.    
\end{claim}

\begin{proof}
Suppose that  
$s\in H^o_i$,  $s=pb$ where $p, b\in S^o_i$, $p$ is plain, $b$ is bold.
Let $r\in Desc(s)$.
Then there is a sequence $s = s_0, s_1, \ldots, s_{j} = r$
such that for every $0\le k\le j-1$, $s_k\in H^o_{i+k}$, and 
$s_{k}$ is a parent of $s_{k+1}$.
For every $0\le k\le j$, let 
$\bar{s}_k=\bar{p}_k\bar{b}_k$ 
be the {\em  extension} of $s_k$. For the exact definition see Definition~3~(b).

Let $s_1=b_1p_1$ where $b_1, p_1\in S^o_{i+1}$,   $p_1$ is plain, $b_1$ is bold.
Move $s_1$ horizontally to the right by $\varepsilon^o_{i+1}/4$, let 
 $s'_1=b'_1p'_1$ be the resulting segment and let 
$\bar{s}_1'=\bar{p}_1'\bar{b}_1'$ be the {extension} of $s'_1$.
Then $b'_1$ is exactly the midpoint of $s$.
The point $p'_1$ has the same $y$-coordinate as either $p$ or $b$, 
and their (horizontal) distance is at most $d^o_{i+1}-\varepsilon^o_{i+1}/4$.
Consequently,
$|\bar{p}_0\bar{p}_1'|\le 2^{i+1}(d^o_{i+1}-\varepsilon^o_{i+1}/4)$ 
so $|\bar{p}_0\bar{p}_1|\le 2^{i+1}d^o_{i+1}.$
Similarly, 
$|\bar{b}_0\bar{b}_1'|\le 2^{i+1}(d^o_{i+1}-\varepsilon^o_{i+1}/4)$ so 
$|\bar{b}_0\bar{b}_1|\le 2^{i+1}d^o_{i+1}.$

We can show by the same argument that for any $1\le k\le j$,
$|\bar{p}_{k-1}\bar{p}_{k}|\le 2^{i+k}d^o_{i+k}$.
and 
$|\bar{b}_{k-1}\bar{b}_{k}|\le 2^{i+k}d^o_{i+k}$.

Therefore, $|\bar{p}_0\bar{p}_{j}|\le 
\sum_{k=1}^{j}2^{i+k}d^o_{i+k}<2^{i+2}d^o_{i+1}=\alpha$ and
$|\bar{b}_0\bar{b}_{j}|\le \sum_{k=1}^{j}2^{i+k}d^o_{i+k}<2^{i+2}d^o_{i+1}=\alpha$. 
\end{proof}

\begin{figure}
	
			\begin{center}
			\begin{tikzpicture}[>=stealth', scale=2]
			
		\fill [ fill=szurke] (-0.5,0) -- (0.5,0) -- (2.5,3) -- (1.5,3)  -- (-0.5,0);
		 \draw[ dotted] (0.5,0) -- (2.5,3);
	 	\draw[ dotted] (1.5,3)  -- (-0.5,0);
		 \draw[ dotted] (1,0) -- (3,3);
	 	\draw[ dotted] (1,3)  -- (-1,0);
			 \tikzstyle{b}=[draw,circle,fill=black,minimum size=3,inner sep=0]
			 \tikzstyle{t}=[draw,circle,fill=white,minimum size=3,inner sep=0]
 
			 \coordinate  (b) at (0,0) {};
			 \coordinate  (p) at (2,3) {};
		 	\coordinate  (pxx) at (0.2,0)  {};
		 	\coordinate  (bxx) at (1.7,3)  {};
			 
	 		\draw[thin, dashed] (bxx) -- (pxx) node[pos=0.5, right] {};
	 			 \draw[ dashed] (b) -- (p) node[pos=0.5, right] {};
	 		\node[b] (px) at ($(p)!0.7! (b)$) {};
			\node[t] (bx) at ($(p)!0.3! (b)$)   {};

	 		\node[t] (pxy) at ($(pxx)!0.7! (bxx)$) {};
			\node[b] (bxy) at ($(pxx)!0.53! (bxx)$)   {};
		
			 \draw[ thick] (bx) -- (px) node[pos=0.5, right] {$s$};
			 \draw[ ] (bxy) -- (pxy) node[pos=0.5, left] {$r$};
			 \draw[ dashed] (-1.5,3) -- (3.5,3) node[pos=0.1, above] {$y = 1$};;
			 \draw[ dashed] (-1.5,0) -- (3.5,0) node[pos=0.1, below] {$y = 0$};;
 
		\node[b] (bx) at (2.5,1.5) [label=right:$q$] {};

			\end{tikzpicture}
			\end{center}\caption{Point $r$ is a descendant of $s$, but $q$ is not.}\label{fig:strip2}
	
\end{figure}
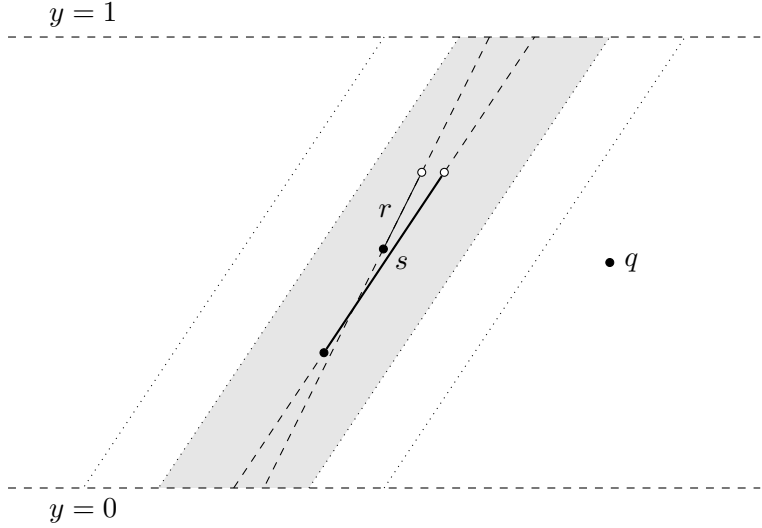

\begin{claim}\label{stripproperty2}
(a) Let $j\ge i\ge 0$ and $s\in H^o_i$. 
Suppose that $q\in S^o_j$ is not an endpoint of $s$ and not a descendant of $s$. 
Then 
$q\not\in S_{2\alpha}(s)$ where $\alpha=2^{i+2}d^o_{i+1}$. 

(b) Suppose that $q_i\in S^o_i$, $s_i\in H^o_i$ and $q_i$ is not an endpoint of $s_i$.
Let $q_{i+1}\in S^o_{i+1}$ be a child of $q_i$, 
$s_{i+1}\in H^o_{i+1}$ be a child of $s_i$.

Then $q_i$ is below the line of $s_i$ if and only if 
$q_{i+1}$ is below the line of $s_{i+1}$.

\end{claim}
\begin{proof}

First we show part (a) of the claim.
Let $j\ge i$, 
$q\in S^o_j$ that is not an endpoint of $s$ or 
any descendant of $s$.


\smallskip

\noindent {\bf Case 1.} {\em Suppose that $q$ is bold.}
There is a sequence $b_0, b_1, \ldots, b_j=q$ such that
for every $k$, $0\le k\le j$, $b_k$ is a bold point in $S^o_k$, and for 
every $k$, $0\le k<j$,
$b_{k}$ is a parent of $b_{k+1}$.
Moreover, there is a sequence 
$s'_0, s'_1, \ldots, s'_{j-1}$ such that
for every $k$, $0\le k<j$, $s'_k\in H^o_k$, 
$b_{k+1}$ is {\em assigned to} $s'_k$, and 
for 
every $k$, $0\le k<j-1$,
$s'_{k}$ is a parent of $s'_{k+1}$.

Similarly, there is a sequence 
$s_0, s_1, \ldots, s_i=s$ such that
for every $k$, $0\le k\le i$, $s_k\in H^o_k$, and for 
every $k$, $0\le k<j$,
$s_{k}$ is a parent of $s_{k+1}$.

By the assumption, if $j>i$, then 
$s'_i\neq s_i$,
if $j=i$ then 
$s'_{i-1}\neq s_{i-1}$. 
Let $m\le i$ be the smallest number with the property that 
$s'_m\neq s_m$.

\smallskip

{\bf Case 1.1} 
{\em Suppose first that $m=i$.} Then $s'_i\neq s_i$ but $s'_{i-1}=s_{i-1}$. 
In this case both $s_i$ and $s'_i$ 
are children of $s_{i-1}=s'_{i-1}$ and  $j>i$ by the previous observations. 
Let $\ell$ be the horizontal line through $q=b_j$. We will analyze the
situation step by step on that line. 
Let $t$ be the intersection of $\ell$ and the line 
of $s_i$.
For $k=i, \ldots, j-1$, let $t_k$ be the intersection of 
$\ell$ and the line of $s'_k$ and let $t_j=q$.
See Figure~\ref{fig:claim23}.

Clearly, $|tt_i|\ge 2^{i-j-1}\varepsilon^o_{i}$.
For any $k$, $i\le k<j$, $|t_kt_{k+1}|\le d^o_{k+1}$.
But 
$$\sum_{k=i}^{j-1}|t_kt_{k+1}|\le \sum_{k=i}^{j-1}d^o_{k+1}\le 2d^o_{i+1},$$
therefore, for $\alpha=2^{i+2}d^o_{i+1}$, 
$$|tq|\ge |tt_i|-\sum_{k=i}^{j-1}|t_kt_{k+1}|\ge 2^{i-j-1}\varepsilon^o_{i}-2d^o_{i+1}$$
$$=2^{i-j-1}2^{-(4o+4)i}-2^{-(4o+3)(i+1)+1}\ge 2^{i-o-1}2^{-(4o+4)i}-2^{-(4o+3)(i+1)+1}$$
$$>2^{-(4o+3)(i+1)+i+3}=2\alpha.$$

Consequently, $q\not\in S_{2\alpha}(s)$.
Here we used that $o\ge j$.
\smallskip

{\bf Case 1.2} {\em Suppose now that $i>m$.}
By the previous argument,
$q\not\in S_{2\beta}(s_m)$ where $\beta=2^{m+2}d^o_{m+1}$. 
By Claim \ref{stripproperty1}, $\bar{s}_i\subset S_{\beta}(s_m)$, therefore,
$S_{2\alpha}(s_i)\subset S_{2\beta}(s_m)$, so it follows that 
$q\not\in S_{2\alpha}(s_i)$.

\medskip
	\begin{figure}	
				\begin{center}
				\begin{tikzpicture}[>=stealth', scale=6,  spy using outlines={rectangle,lens={scale=3}, size=8cm, connect spies}]

				 \tikzstyle{pl}=[draw,circle,fill=white,minimum size=4,inner sep=0]
				 \tikzstyle{bo}=[draw,circle,fill=black,minimum size=4,inner sep=0]
				 \tikzstyle{p}=[draw,circle,fill=black,minimum size=2,inner sep=0]
 
				 \coordinate  (b) at (0,0) {};
				 \coordinate  (p) at (2,3) {};
				 \node[pl]  (pb) at (0.9,1.5) {};
		 		\node[pl] (px) at ($(p)!1! (b)$) {};
				\node[bo] (bx) at ($(p)!0.5! (b)$)   {};
 
		\coordinate  (pxx) at (-0.8,0)  {};
		\coordinate  (bxx) at (2.8,3)  {};

		 		\node[pl] (pxy) at ($(pxx)!1! (bxx)$) {};
				\node[bo] (bxy) at ($(pxx)!0.5! (bxx)$)   {};
				\node[p] at (1.375,3 - 2.5*3/8) [label=above left:{\scriptsize $t$}] {};
				\node[p] (q) at (1.585,3 - 2.5*3/8)  [label=above:{\scriptsize ${q}$}] {};
				\node[p] (q1) at (1.635,3 - 2.5*3/8) {};

				\node[bo] (bxyx) at ($(pxy)!0.5! (bxy) + (-0.02,0)$)   {};
				\node[pl] (bnxyx) at ($(pxy)!0.5! (bxy) + (-0.1,0)$)   {};
				\node[p] (bxyxx) at ($(pxy)!0.625! (bxy)$) [label=below :{\scriptsize ${t_i}$}]   {};
				\node[bo] (bnx) at ($(bxyx)!0.5! (pb) + (-0.02,0)$)   {};
				 \draw[ dashed] (bx) -- (p) node[pos=0.8, left] {};
				 \draw[ ] (bx) -- (px) node[pos=0.7, right] {$s_i$};
				 \draw[ ] (bxy) -- (pxy) node[pos=0.6, right] {$s'_i$};
				  \draw[ ] (bxyx) -- (pb) node[pos=0.55, right] {$s'_{i+1}$};
				  \draw[ ] (bnx) -- (bnxyx) node[pos=0.3, left] {\scriptsize ${s'_{i+2}}$};
				 \draw[ dotted] (0,3 - 2.5*3/8) -- (2.8,3 - 2.5*3/8) node[pos=0.1, above] {$\ell$};

		 \spy [height=6cm,width=8cm,magnification=3,connect spies] on (9,12.4)
		             in node [left] at (2.5,0.7);
		 
				\end{tikzpicture}
				\end{center}\caption{Case 1.1 in the proof of Claim 2.2.}\label{fig:claim23}
		\end{figure}
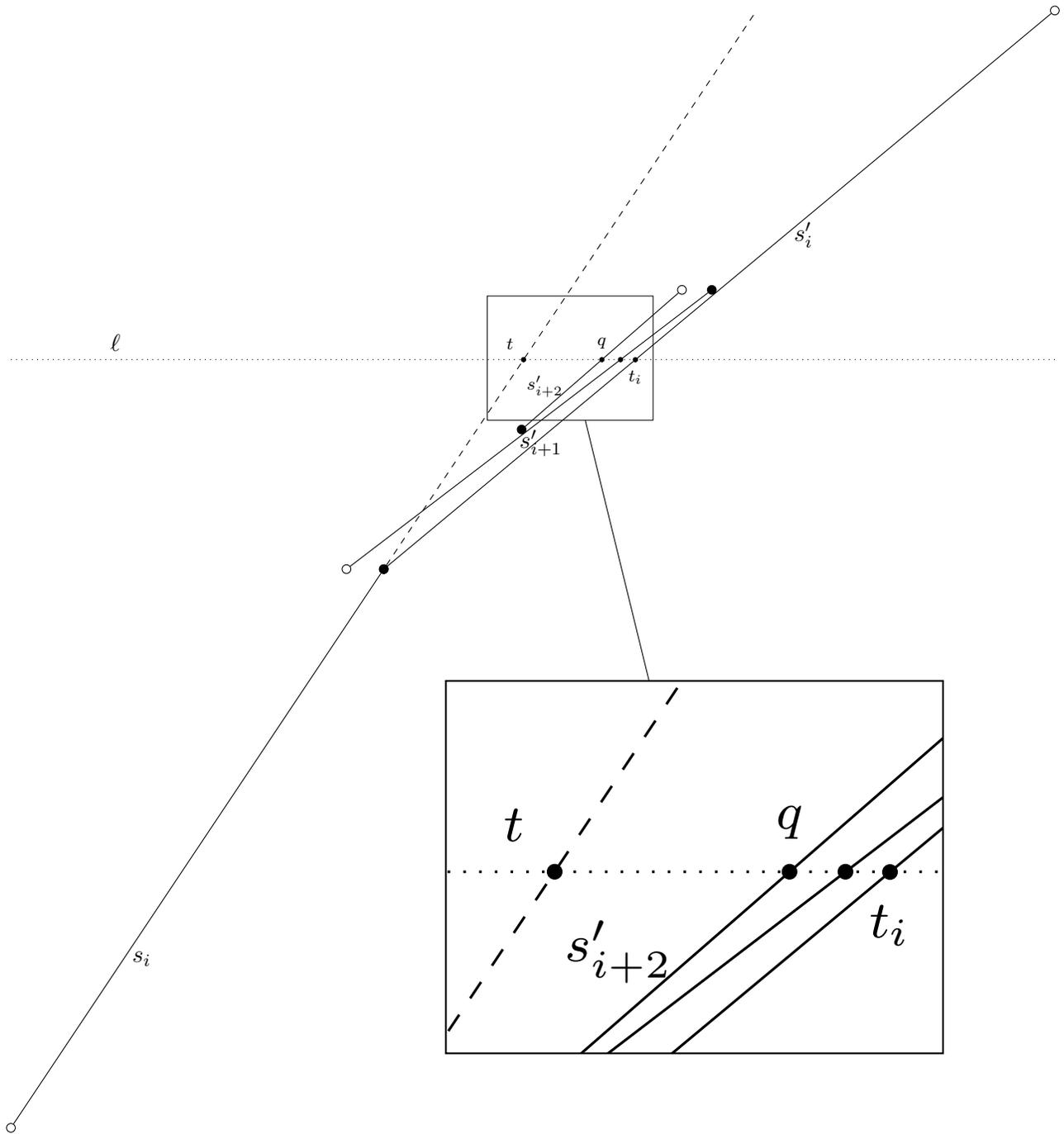



\noindent  {\bf Case 2} {\em Suppose that $q$ is plain.}
The argument will be a more complicated version of the previous one.
There is a sequence $b_0, b_1, \ldots, b_l, p_{l+1}, \ldots,  p_j=q$ such that

\noindent (a) for every $k$, $0\le k\le l$, $b_k$ is a bold point in $S^o_k$, 

\noindent (b) for every $k$, $0\le k<l$, $b_{k}$ is a parent of $b_{k+1}$, 

\noindent (c) for 
every $k$, $l<k\le j$, $p_k$ is a plain point in $S^o_k$, 

\noindent (d) for 
every $k$, $l<k<j$,
$p_{k}$ is a parent of $p_{k+1}$, and 

\noindent (e) $b_{l}$ is a parent of $p_{l+1}$.

Moreover, there is a sequence 
$s'_1, s'_2, \ldots, s'_{l-1}$ such that
for every $k$, $0\le k\le l-1$, $s'_k\in H^o_k$, for 
every $k$, $0\le k<l-1$,
$s'_{k}$ is a parent of $s'_{k+1}$, and 
$b_{k+1}$ is {\em assigned to} $s'_k$.

Similarly, there is a sequence 
$s_0, s_1, \ldots, s_i=s$ such that
for every $k$, $0\le k\le i$, $s_k\in H^o_k$, and for 
every $k$, $0\le k<i$,
$s_{k}$ is a parent of $s_{k+1}$.



\smallskip

{\bf Case 2.1} {\em Suppose first that $l\ge i$.}
We proceed almost exactly as in Case 1.

By the assumption, if $l>i$, then 
$s'_i\neq s_i$,
and if $l=i$ then 
$s'_{i-1}\neq s_{i-1}$. 
Let $m\le i$ be the smallest number with the property that 
$s'_m\neq s_m$.

Then $s'_m\neq s_m$ but both of them are children of $s_{m-1}=s'_{m-1}$. 
Let $\ell$ be the horizontal line through $q$. 
Observe that   $b_l, p_{l+1}, \ldots,  p_j=q$ are all on $\ell$.
Let $t$ be the intersection of $\ell$ and the line of $s_m$.
For $m\le k\le l-1$, let $t_k$ be the intersection of $\ell$ and the line of
$s'_k$, 
let $t_l=b_l$, and for $k=l+1, \ldots, j$, let 
$t_k=p_k$.


Clearly, $|tt_m|\ge 2^{m-l-1}\varepsilon^o_{m}$.
For any $k$, $m\le k< j$, $|t_kt_{k+1}|\le d^o_{k+1}$.
But 
$$\sum_{k=m}^{j-1}|t_kt_{k+1}|\le \sum_{k=m}^{j-1}d^o_{k+1}\le 2d^o_{m+1},$$
therefore, for $\beta=2^{m+2}d^o_{m+1}$, 
$$|tq|\ge |tt_m|-\sum_{k=m}^{j-1}|t_kt_{k+1}|\ge 2^{m-l-1}\varepsilon^o_{m}-2d^o_{m+1}\ge 2\beta.$$
Consequently, $q\not\in S_{2\beta}(s_m)$.
By Claim \ref{stripproperty1}, $\bar{s}_i\in S_{\beta}(s_m)$. 
Recall, that  $\alpha=2^{i+2}d^o_{i+1}$. 
Therefore,
$S_{2\alpha}(s_i)\subset S_{2\beta}(s_m)$, so it follows that 
$q\not\in S_{2\alpha}(s_i)$.

\smallskip

{\bf Case 2.2} {\em Suppose now that $l< i$.}
If $s'_{l-1}\neq s_{l-1}$, then we can proceed exactly as in Case 2.1. 
Let $m\le i$ be the smallest number with the property that 
$s'_m\neq s_m$, and do the same calculation. 

So we can assume that $s'_{l-1}=s_{l-1}$. In this case $b_l$ is an endpoint of
$s_l$. For simplicity denote $b_l$ also as $p_l$. 
By the assumption, $p_i$ is not an endpoint of $s_i$.
Let $m$ be the smallest number such that $p_m$ is not the endpoint of $s_m$. 
Now we do the same calculation again.

Let $\ell$ be the horizontal line through $p_m, \ldots, p_j$.  
Let $t$ be the intersection of $\ell$ and the line of $s_m$.
We have $|tp_m|\ge\varepsilon^o_{m}/2$.
For any $k$, $m\le k< j$, $|p_kp_{k+1}|\le d^o_{k+1}$.
So
$$\sum_{k=m}^{j-1}|p_kp_{k+1}|\le \sum_{k=m}^{j-1}d^o_{k+1}\le 2d^o_{m+1},$$
therefore, for $\beta=2^{m+2}d^o_{m+1}$, 
$$|tq|\ge |tp_m|-\sum_{k=m}^{j-1}|p_kp_{k+1}|\ge\varepsilon^o_{m}/2-2d^o_{m+1}\ge 2\beta.$$
Consequently, $q\not\in S_{2\beta}(s_m)$.
By Claim \ref{stripproperty1}, $\bar{s}_i\subset S_{\beta}(s_m)$, therefore,
$S_{2\alpha}(s_i)\subset S_{2\beta}(s_m)$, so it follows that 
$q\not\in S_{2\alpha}(s_i)$.

This concludes the proof of part (a). Part (b) follows directly from the
calculations, we only have to observe that taking a child of the 
point $q$, it cannot 
``jump'' to the other side of the strip $S_{2\alpha}(s)$.
\end{proof}

\bigskip

We are ready to prove that 
the segments in $H^o_i$, $i\le o$,
 are indeed halving segments. Let $o>0$ be fixed. 
First we show that the new lines are ``locally'' halving. 
Suppose that $s=pb\in H^o_{i-1}$ where $p, b\in S^o_{i-1}$, $p$ is plain, $b$ is
bold. To construct $S^o_i$ and $H^o_i$, we replace
$p$ with the arithmetic progression 
$p^{(0)}, \ldots,  p^{(a_i-1)}$ of plain points, replace $b$ with
the arithmetic progression 
$b^{(0)}, \ldots,  b^{(a_i)}$ of plain points and add the bold
point $b'$ close to the midpoint of $pb$ as
described before. Then for any $k$, the line $b'p^{(k)}$ has 
$p^{(0)}, \ldots, p^{(k-1)}, b^{(k+1)}, \ldots,  b^{(a_i)}$ on one side and 
$b^{(0)}, \ldots,  b^{(k)}, p^{(k+1)}, \ldots,  p^{(a_i-1)}$ on the other side. 
We can argue the same way for the line $b'b^{(k)}$. 
So, these lines halve the set
of plain points that replace $p$ and $b$. 
Call this the {\em locally halving property} of the halving segments.

For every $i$, let $B_i$ resp. $P_i$ be the set of bold resp. plain
points of $S^o_i$.
We prove by induction that every  $s\in H^o_i$
halves both of the sets $B_i$ and $P_i$. 
More precisely, the two open halfplanes determined by the line of $s$ contain
the same number of points of 
 $B_i$ 
and they also contain the same number of points of $P_i$.
For $i=0, 1$ it  is trivial. Assume that
$s_{i-1}=(p_{i-1}, b_{i-1})\in H^o_{i-1}$
halves the sets $B_{i-1}$ and $P_{i-1}$. 

The set $B_i=B_i^{(1)}\cup B_i^{(2)}$ where 
$B_i^{(1)}$ is the set of bold children of points of $B_{i-1}\setminus\{b_{i-1}\}$ and  
$B_i^{(2)}$ is the set of bold children of $b_{i-1}$. 

Similarly, 
$P_i=P_i^{(1)}\cup P_i^{(2)}\cup P_i^{(3)}$
where
$P_i^{(1)}$ is the set of (plain) children of points of $P_{i-1}\setminus\{p_{i-1}\}$,
$P_i^{(2)}$ is the set of plain children of points of $B_{i-1}\setminus\{b_{i-1}\}$, and
$P_i^{(3)}$ is the set of plain children of $p_{i-1}$ and $b_{i-1}$. 

Let $s_i$ be a child of $s_{i-1}$.
It follows from Claim \ref{stripproperty2}
that if $p\in P_{i-1}$ is on the left (resp. right) side of the line of $s_{i-1}$, then
any descendant $p'\in P_i$ of $p$ is on the left (resp. right) side of $s_{i}$. 
Since each  $p\in P_{i-1}$ is replaced by the same number, $a_i$ of plain points, it follows that
$s_i$ is halving $P_i^{(1)}$.
By a similar argument, $s_i$ is halving $P_i^{(2)}$.
And by the locally halving property of $s_i$, it is halving the set  $P_i^{(3)}$.
So, $s_i$ is halving $P_i$.

Since $s_{i-1}$ is halving $B_{i-1}$, and each bold point has exactly $2a_{i-1}+1$ bold children, 
it follows again from Claim \ref{stripproperty2} that $s_i$ is halving the set $B_i^{(1)}$.
The set $B_i^{(2)}$ contains $2a_{i-1}+1$ bold points. 
It follows from the locally halving property of $s_{i-1}$ and  Claim \ref{stripproperty2}
that $s_i$ is halving the set $B_i^{(2)}$. Therefore, $s_i$ is halving $B_i$ as well.
This concludes the proof 
that the segments in $H^o_i$ are indeed halving segments for every $o\ge i$.


\smallskip

For any $o$, the set $S_0^o$ contains only two points, 
$(1,1)$ and $(0,0)$, and the line determined by them has slope $1$. 
We show that for any $o\ge i$, any two points of $S_i^o$ determine 
a horizonal line, or line of slope close to $1$.
This will be very useful in the proof of part (c) of Lemma 1. 

\begin{claim}\label{szog}
Suppose that 
$\ell$ is a line determined by two points of $S_i$ and let $\gamma$ be the
smaller angle determined by $\ell$ and the $x$-axis. 
Then either $\gamma=0$ or 
$7/8\le\cot\gamma\le 9/8$.
\end{claim}

\begin{proof}
For any point $p$, let $x(p)$ and $y(p)$ denote its $x$-coordinate and $y$-coordinate, respectively. 
The statement of Claim \ref{szog} is trivial for $o=0$. Note that in this
case $i=0$ as well, since $o\ge i\ge 0$. 
Let $o>0$ and let $o\ge i\ge 0$. 
Let $a, b\in S_{i}^o$. If $y(a)=y(b)$, then they determine a horizontal line.
Suppose that  $y(a)\neq y(b)$. Then we have $|y(a)-y(b)|\ge 2^{-i}\ge 2^{-o}$.
Let $\gamma$ be the
smaller angle determined by the line $ab$ and the $x$-axis. 
Then $$\cot(\gamma)=(x(a)-x(b))/(y(a)-y(b)).$$
The set $S_0^o$ contains only two points, 
$(1,1)$ and $(0,0)$, and they determine  $s_0\in H_0^o$.
Apply Claim \ref{stripproperty1} for the segment $s_0$,
we obtain that $a, b\in S_{\alpha}(s_0)$ where $\alpha=4d_1^o=2^{-4o-1}$.
Therefore, $|y(a)-x(a)|\le 2^{-4o-1}$ and 
$|y(b)-x(b)|\le 2^{-4o-1}$. Therefore,
$$|1-\cot(\gamma)|\le \frac{|y(a)-x(a)|+|y(b)-x(b)|}{|y(a)-y(b)|}\le \frac{2^{-4o}}{2^{-o}}=2^{-3o}\le 1/8.$$
\end{proof}

Note that $7/8\le\cot\gamma\le 9/8$ implies that 
$4\pi/18<\gamma<5\pi/18$.


It is clear 
from the construction, that for every $o, o'\ge i$, 
$G^o_{i}(S^o_{i}, H^o_{i})$ and 
$G^{o'}_{i}(S^{o'}_{i}, H^{o'}_{i})$
represent the same {\em abstract} graph, but they are different as geometric
graphs.
For any $i$, let 
$$G_{i}(S_{i}, H_{i})=G^i_{i}(S^i_{i}, H^i_{i}).$$
The graph $G_i$ has $n_i$ vertices and $m_i$ edges, all of them are halving edges of $S_i$.

\bigskip

We estimate now the number of points $n_i$ in $S_i$ and the number of halving
lines $m_i$ in $H_i$, that is, we prove parts (a) and (b) of Lemma 1. 
Our calculation is similar to the one by
Nivasch \cite{N08}. 

We have $n_0=2$ and $m_0=1$. 
Since each halving line in $H_{i-1}$ is replaced by $2a_i+1$ halving lines in
$H_i$,
$$m_i=(2a_i+1)m_{i-1},$$
therefore,
$$m_i=m_0\prod_{j=1}^{i}(2a_j+1)=\prod_{j=2}^{i+1}(2^j+1)=
\frac{1}{6}\prod_{j=0}^{i+1}(2^j+1).$$
Consequently
$$m_i>\frac{1}{6}\prod_{j=0}^{i+1}2^j=\frac{1}{6}2^{i^2/2+3i/2+1}= \frac{1}{3}2^{i^2/2+3i/2}.$$
On the other hand,
$$m_i=\frac{1}{6}\prod_{j=0}^{i+1}2^j\cdot\prod_{j=0}^{i+1}\left(\frac{2^j+1}{2^j}\right)=
\frac{1}{3}2^{i^2/2+3i/2}\cdot\prod_{j=0}^{i+1}\left(1+2^{-j}\right)$$
$$<\frac{1}{3}2^{i^2/2+3i/2}\cdot\prod_{j=0}^{i+1}\left(e^{2^{-j}}\right)<\frac{e^2}{3}2^{i^2/2+3i/2}.$$
Here we used that for any $x>0$, $e^x>1+x$.

Summarizing,
\begin{equation}\label{m}
\frac{1}{3}2^{i^2/2+3i/2}<m_i<\frac{e^2}{3}2^{i^2/2+3i/2}.
\end{equation}

The number of bold points in $S_{i-1}$ is $m_{i-2}$, hence the number of plain
points in $S_{i-1}$  
is $n_{i-1}-m_{i-2}$. Therefore, there are 
$$a_i(n_{i-1}-m_{i-2})+(a_i+1)m_{i-2}=a_in_{i-1}+m_{i-2}$$
plain points and $m_{i-1}$ bold points in $S_i$, so
$$n_i=a_in_{i-1}+m_{i-1}+m_{i-2}.$$
Using (\ref{m}) for $m_i$, we prove by induction that
\begin{equation}\label{n}
2^{i^2/2+i/2}<n_i<4(i+1)2^{i^2/2+i/2}.
\end{equation}
Both inequalities hold trivially for $i=0$.
Suppose that 
$i>0$ and 
$2^{(i-1)^2/2+(i-1)/2}<n_{i-1}$.
Then
$$n_i=a_in_{i-1}+m_{i-1}+m_{i-2}>a_in_{i-1}>2^i2^{(i-1)^2/2+(i-1)/2}=2^{i^2/2+i/2}.$$
Suppose now that $n_{i-1}<4i2^{(i-1)^2/2+(i-1)/2}$.
Then
$$n_i=a_in_{i-1}+m_{i-1}+m_{i-2}<a_in_{i-1}+2m_{i-1}$$
$$<4i2^i2^{(i-1)^2/2+(i-1)/2}+\frac{2e^2}{3}2^{(i-1)^2/2+3(i-1)/2}<4(i+1)2^{i^2/2+i/2}.$$

It follows that $m_i\ge n_ie^{\Omega\left({\sqrt{\log n_i}}\right)}$.
This finishes the proof of parts (a) and (b) of Lemma 1.


Now we prove part (c). The statement is trivial for $i=0$, suppose that $i>0$.
Let $a, b\in S_i$. 
First assume that $y(a)\neq y(b)$. Then we have 
$2^{-i}\le |y(a)-y(b)|\le 1$, so
by Claim \ref{szog},
$$7\cdot 2^{-i}/8\le |x(a)-x(b)|\le 9/8.$$
Now assume that $y(a)=y(b)$. 
Then $|x(a)-x(b)|<\sum_{j=1}^{i}d^i_j=\sum_{j=1}^{i}2^{-(4i+3)j}<1$,
on the other hand, $|x(a)-x(b)|\ge \varepsilon^i_i=2^{-4i^2-4i}$.
Summarizing, in any case we have
$$9/8\ge |x(a)-x(b)|\ge 2^{-4i^2-4i}\ge n_i^{-8}$$
by inequality (\ref{n}). Apply a scaling by a factor of $8/9$  
and the statement follows. 
This concludes the proof of Lemma 1. 
\hfill\(\qedsymbol\)

\smallskip

\section{Second construction}

The following statement is stronger than Lemma 1.
It holds for all even numbers $n$, not just a sequence $\{ n_i\}$, and 
in part (c), instead of $\Omega(n_i^{-8})$, 
now we have $\Omega(n^{-1})$, which is optimal.

\bigskip

 



 



\noindent {\bf Lemma 2.} {\em For every even $n>0$ there is a planar point set
$S^*(n)$ of $n$ points and $m=ne^{\Omega\left({\sqrt{\log n}}\right)}$ halving
  lines
such that 
for any two points of $S(n)$, the difference of their 
$x$-coordinates is $\Omega(n^{-1})$ and 
at most $1$.}

\bigskip

\noindent {\bf Proof of Lemma 2.}
Our construction is based on 
$G_{i}(S_{i}, H_{i})$ from Lemma 1. Assume without loss of generality that
$i>10$. This will slightly simplify our calculations.

\smallskip

\begin{claim}\label{mozgatas}
Consider the set $S=S_{i}$ of $n_i$ points from Lemma 1. 
Move each of its points horizontally by 
a distance at most $n_i^{-9}$. 
Let  
$S'$ be the resulting point set. 
If $x, y\in S$ and $xy\in H_i$ then $x'$ and $y'$, the corresponding points in $S'$,
determine a halving line of $S'$. 
\end{claim}

\begin{proof}
Let $s=pb$, $p, b\in S_i$, $s\in H_i$, and let $q\in S_i$, $q\neq p, b$. 
By Claim \ref{stripproperty2}, $q\not\in S_{2\alpha}(s)$ where $2\alpha
=2^{i+3}d^i_{i+1}=2^{-4i^2-6i}$.
On the other hand $|y(p)-y(b)|\ge 2^{-i}$. 
Move $p$, $b$, and $q$ horizontally by 
at most $n_i^{-9}$. The horizontal distance of the line of $s$ and the point
$q$ changed by at most $d=n_i^{-9}+2^{i+1}n_i^{-9}$.
By inequality (\ref{n}), $n_i\ge 2^{i^2/2+i/2}$, therefore,
$$d=n_i^{-9}+2^{i+1}n_i^{-9}\le
1/2^{9i^2/2+9i/2}+2^{i+1}/2^{9i^2/2+9i/2}$$
$$=1/2^{9i^2/2+9i/2}+2^{i+1}/2^{9i^2/2+7i/2-1}<
2^{-4i^2-6i}=2\alpha.$$ 
We used here that $i>10$. Therefore, $q$ remains on the same side of $s$ after
the perturbation of the points. This holds for any  $q\neq p, b$, therefore, 
$s$ remains a halving line. 
\end{proof}

Suppose now that $i>10$ is a fixed number. 
Let $n=n_i$, $m=m_i$, $S=S_i$, and $H=H_i$.  
Let $p_1, \ldots, p_n$ be the points of $S$, $p_j=(x_j, y_j)$.
Then let $p'_j=(x'_j, y_j)$, where
$$x'_j=\frac{\lfloor n^9x_j\rfloor}{n^9} + \frac{j}{n^{10}}.$$
Let $S'=\{p'_j\ |\ 1\le j\le n\}$ and let $H'$ be the segments corresponding
to the segments in $H$. Each point of $S$ is moved horizontally by less than
$n^{-9}$,
so, by Claim \ref{mozgatas}, the segments in $H'$ are halving segments of $S'$.
Now apply a horizontal translation so that all points of $S'$ have $x$-coordinates 
between $1$ and $3$. This is possible by Lemma 1 (c). 


Let $\delta>0$ be a very small number.
Apply the transformation 
$(x, y)\longrightarrow (x, \delta^2\cdot y)$ on $S'$, call the 
resulting point set $S''$. 
See Figure~\ref{fig:second1}. Clearly, halving edges remained halving edges.
Since $\delta$ is very small, all halving edges are very close to
the $x$-axis and almost parallel to it. All points of $S''$ have very small
$y$-coordinates and 
their $x$-coordinates are between $1$ and $3$. 

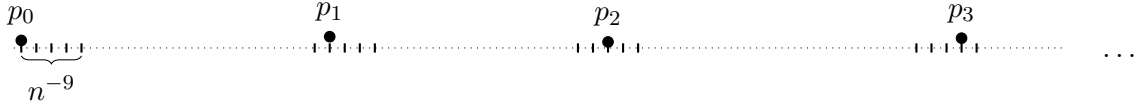
\begin{figure}
\begin{center}
\begin{tikzpicture}[] 

	 \tikzstyle{t}=[draw,circle,fill=black,minimum size=4,inner sep=0]

	\newcommand{\deltanull}{0.2}
	 \draw [dotted](0,0) -- (14,0);

	 \node[t] (p0) at (0.1,0.1) [label=above:$p_0$] {};
	  \node[t] (p1) at (4 + \deltanull,0.15) [label=above:$p_1$] {};
	  \node[t] (p2) at (7.5 + 2*\deltanull,0.08) [label=above:$p_2$] {};
	  \node[t] (p3) at (12 + 3*\deltanull,0.13) [label=above:$p_3$] {};

	\newcounter{ct}
	\forloop{ct}{0}{\value{ct} < 5}%
	{%
	\draw[thick] ([yshift=1.75]$(0.1+\deltanull*\arabic{ct}, 0)$) -- ([yshift=-1.75]$(0.1+\deltanull*\arabic{ct}, 0)$);
	}

	\forloop{ct}{0}{\value{ct} < 5}%
	{%
	\draw[thick] ([yshift=1.75]$(4+\deltanull*\arabic{ct}, 0)$) -- ([yshift=-1.75]$(4+\deltanull*\arabic{ct}, 0)$);
	}

	\forloop{ct}{0}{\value{ct} < 5}%
	{%
	\draw[thick] ([yshift=1.75]$(7.5+\deltanull*\arabic{ct}, 0)$) -- ([yshift=-1.75]$(7.5+\deltanull*\arabic{ct}, 0)$);
	}

	\forloop{ct}{0}{\value{ct} < 5}%
	{%
	\draw[thick] ([yshift=1.75]$(12+\deltanull*\arabic{ct}, 0)$) -- ([yshift=-1.75]$(12+\deltanull*\arabic{ct}, 0)$);
	}
	 
	  \draw [decorate,decoration={brace,mirror, amplitude=3pt}] ([yshift=-3]$(0.1,0)$) -- ([yshift=-3]$(0.9,0)$) node [pos=0.5, yshift=-12pt] {$n^{-9}$};
	 \node at (14.7,-0.1) {$\ldots$};
\end{tikzpicture}
\end{center}\caption{The points of $S''$.}\label{fig:second1}
\end{figure}

Let $\Psi_{\alpha}$ denote the counterclockwise 
rotation about the origin by angle $\alpha$
and let $T_r$ denote the translation to the 
right by $r$.
We have to introduce one more parameter, $N$. Let $n^9\le N\le n^{10}$ be an
arbitrary number.

Define the point sets $S_+^{(k)}$ and
$S_-^{(k)}$, called {\em positive blocks} and {\em
  negative blocks}, respectively, as follows.
For $0\le k\le N$, let
$$S_+^{(k)}= \Psi_{2k\delta}\left(T_{k/n^9}(S'')\right),$$
and for  $0\le k\le N-1$, let
$$S_-^{(k)}= \Psi_{\pi+(2k+1)\delta}\left(T_{k/n^9}(S'')\right).$$
Let $S^{\mbox{\scriptsize union}}$ be their union, that is,
$$S^{\mbox{\scriptsize union}}=\bigcup_{k=0}^{N}S_+^{(k)} \cup \bigcup_{k=0}^{N-1}S_-^{(k)}.$$
See Figure~\ref{fig:second2}.

Finally, apply a scaling of factor $n^9/(6N)$ to $S^{\mbox{\scriptsize union}}$ and let 
$S^{*}$ be a point set obtained.

\medskip

We claim that $S^{*}$ satisfies the conditions. 
Recall that each block has $m$ halving lines, $n=n_i$ and $m=m_i$ 
for some fixed $i$.
The set $S^{*}$ contains $2N+1$ blocks, each contains $n$ points. 
So $|S^{*}|=n^{*}=2Nn+n$.

Observe, that a halving line of a block
has the same number of other blocks on both sides, so it is a halving line of
$S^{*}$ as well.
So, for the number of halving edges of $S^{*}$, 
$$m^{*}=(2N+1)m=(2N+1)n
e^{\Omega\left({\sqrt{\log
      n}}\right)}=n^{*}e^{\Omega\left({\sqrt{\log n^{*}}}\right)}$$
since $N\le n^{10}$.
This proves parts (a) and (b). 
\smallskip

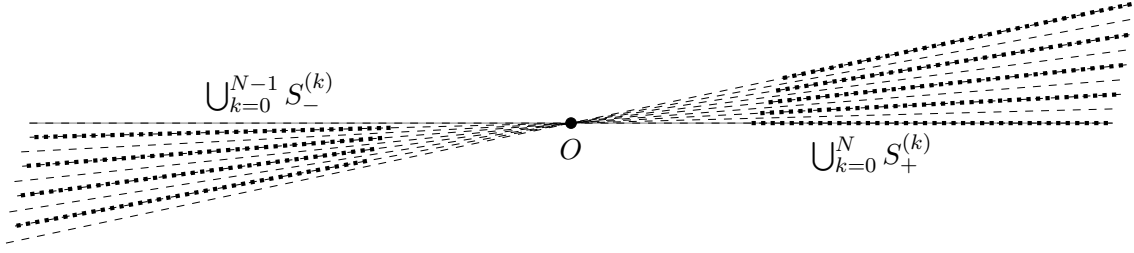
\begin{figure}
 \begin{center}
 \begin{tikzpicture}[] 
	 \begin{scope}[scale=0.8]
 	 \tikzstyle{t}=[draw,circle,fill=black,minimum size=4,inner sep=0]
 	 \draw [gray](-9,0) -- (9,0);
 	 \node[t] (p0) at (0,0) [label=below:$O$] {};
 \newcommand{\deltanullb}{0.15}
 	\forloop{ct}{0}{\value{ct} < 5}%
 	{%
	\begin{scope}[rotate=3*\arabic{ct}]
 	\begin{scope}[shift={(\deltanullb*\arabic{ct},0)}]
 	 \draw [ultra thick, dotted](3,0) -- (9,0);
	 \draw [ultra thin, dashed](-9 - 2* \deltanullb*\arabic{ct},0) -- (9,0);
 	 \end{scope}
	 \end{scope}
 	}

 	\forloop{ct}{0}{\value{ct} < 4}%
 	{%
	\begin{scope}[rotate=180+3*\arabic{ct}+1.5]
 	\begin{scope}[shift={(\deltanullb*\arabic{ct},0)}]
 	 \draw [ultra thick, dotted](3,0) -- (9,0);
	 \draw [ultra thin, dashed](-9 - 2* \deltanullb*\arabic{ct},0) -- (9 ,0);
 	 \end{scope}
	 \end{scope}
 	}
	
		 \node at (5,-0.5) {$\bigcup_{k=0}^{N}S_+^{(k)}$};
		  \node at (-5,0.5) {$\bigcup_{k=0}^{N-1}S_-^{(k)}$};
		  \end{scope}
 \end{tikzpicture}
 \end{center}\caption{The construction of $S^*$.}\label{fig:second2}
\end{figure}

By the construction, the 
$x$-coordinates of 
the points of $S''$ 
are between $1$ and $3$.
Consequently, 
the $x$-coordinates of 
the points of a positive block $S_+^{(k)}$ (resp. negative block $S_-^{(k)}$)
are between $0$ and  $3N/n^9$ (resp. $-3N/n^9$ and $0$).
Therefore, 
the 
$x$-coordinates of 
the points of 
$S^{\mbox{\scriptsize union}}$ 
are between $3N/n^9$ and  $-3N/n^9$.
Finally, we can conclude that 
the 
$x$-coordinates of 
the points of 
$S^{*}$ 
are between $1/2$ and  $-1/2$.

Recall that $p'_1, \ldots, p'_n$ are the points of $S'$,  $p'_j=(x'_j, y_j)$,
and for every $j$, $x'_j=\frac{k}{n^9} + \frac{j}{n^{10}}$ for some 
integer $k$.
This remains true for $S^{''}$ and still remains true
if we apply a translation by $\frac{l}{n^9}$ for some
integer $l$. If we apply now a rotation about the origin by a very small
angle, 
the $x$-coordinates will change by a very small amount.
So, we have the following statement: let  $q_1, \ldots, q_n$ be the points of a
positive block
$S_+^{(k)}$, corresponding to the points $p'_1, \ldots, p'_n$ 
of $S'$. Then the $x$-coordinate $x_j$ of $q_j$ is very close to a number of
the form $\frac{k}{n^9} + \frac{l}{n^{10}}$ for some $k, l$ integers.
Analogous statement holds for the negative blocks. 
So, for any two points of 
$S^{\mbox{\scriptsize union}}$, the difference of their 
$x$-coordinates is least 
$n^{-10}/2$. 
After scaling by a factor of $n^9/6N$ we get the final set $S^{*}$, so in
$S^{*}$, 
the difference of the 
$x$-coordinates of any two points is at least 
$\frac{n}{12N}=\Omega((n^{*})^{-1})$. 
This proves Lemma 2 if $n^{*}$, the number of points, is of the form $(2N+1)n$, 
where
$n=n_i$ for some $i$ and $n^9\le N\le n^{10}$. For other values we have to add
some extra points as 
follows.

\smallskip

Observe, that for every $i>10$, $(2n_i^{10}+1)n_i>(2n_{i+1}^9+1)n_{i+1}$. 
This means, 
that for any large enough even $n$ (say, $n>(2n_{10}^9+1)n_{10}$),
there is an $i>10$ and an $N$,  $n_i^9\le N\le n_i^{10}$, such that
$(2N+1)n_i\le n\le (2N+3)n_i$. So we can take our construction with parameters
$n_i$ and $N$, and add at most $2n_i$ extra points so that the conditions are
still satisfied.
This concludes the proof of Lemma~2. \hfill\(\qedsymbol\)


\section{Third construction}

%
%
%

\noindent {\bf Proof of Theorem 1.}
Let $S^{*}$ be the point set satisfying the conditions of Lemma 2. 
Let $n$ be the number of its points and let $m$ be the number of its halving lines.
Let $\varepsilon>0$ be a very small number.

Apply the transformation 
$(x, y)\longrightarrow (x+1/2, \varepsilon^2\cdot y)$ on $S^{*}$ from Lemma 2, 
and let $R$ denote
the resulting point set. 
The set $R$ contains $n$ points and 
$m=ne^{\Omega\left({\sqrt{\log n}}\right)}$ halving lines. 
The $x$-coordinates of its points are between 1 and 2, the $y$-coordinates are
extremely small, and the distance between any two points is $\Omega(1/n)$.

For $0\le k\le n$, let
$$R_k= \Psi_{2k\pi/(n+1)}(R)$$
and let 
$$P=\bigcup_{k=0}^{n}R_k.$$
See Figure~\ref{fig:third}.

The set $P$ contains $N=n(n+1)$ points. For every $k$,  each halving
line of $R_k$ contains the same number of other blocks on both sides, so it is
a halving line of $P$. Therefore, the number of halving lines of $P$,
$M=(n+1)ne^{\Omega\left({\sqrt{\log n}}\right)}=Ne^{\Omega\left({\sqrt{\log N}}\right)}$.
The minimum distance among the points in $P$ is
$\Omega(1/n)=\Omega(1/\sqrt{N})$ and the diameter of $P$ is less than
$4$. 
This finishes the proof of Theorem~1, if the number of points, $N$, is of the form $n(n+1)$ where $n$ is even.
If we want a construction with $N$ even number of points where 
$n(n+1)<N<(n+2)(n+3)$ for some $n$ even, 
we take the 
construction with 
$n(n+1)$ points and add $N-n(n+1)<4n+6=O(N^{1/2})$ extra points so that the conditions are still satisfied. 
One possible way to do it is that we add the vertices of a regular 
$(N-n(n+1))$-gon, inscribed in a circle of radius $3$ about the origin, so that none of the extra points are 
on the (original) halving lines. 
This concludes the proof. \hfill\(\qedsymbol\)

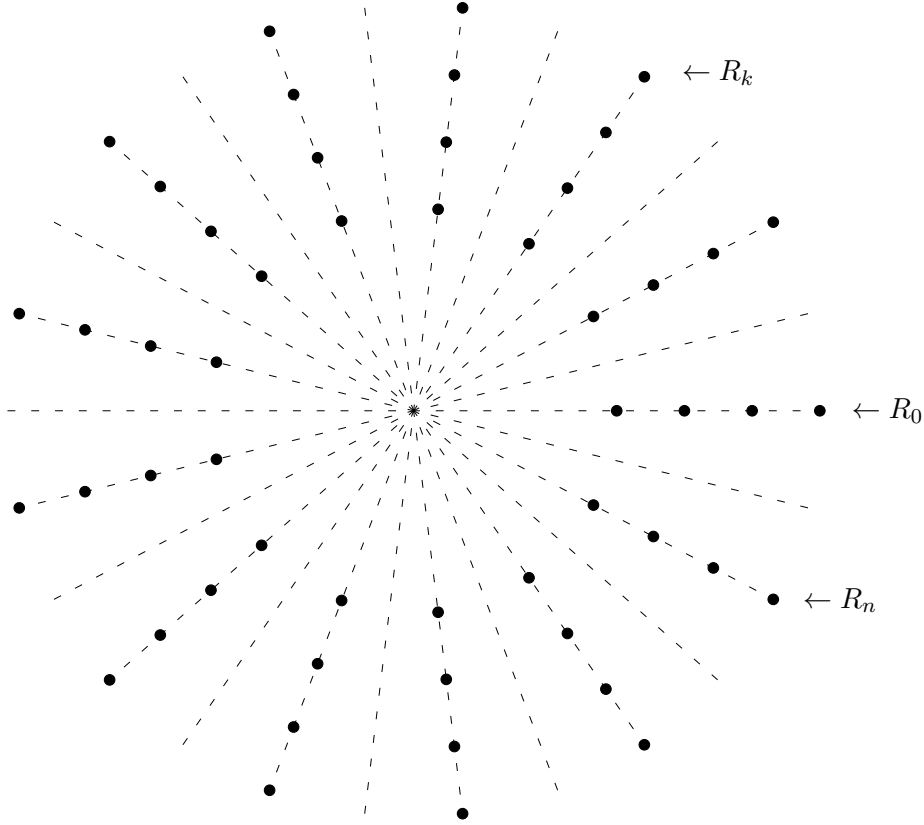
\begin{figure}
	\begin{center}
	\begin{tikzpicture}[scale=0.9] 
		 \tikzstyle{t}=[draw,circle,fill=black,minimum size=4,inner sep=0]
	 
	\newcommand{\deltanullb}{0.15}
		\forloop{ct}{0}{\value{ct} < 13}%
		{%
		\begin{scope}[shift={(0,0)},rotate=(360/13)*\arabic{ct}]
			\node[t] (p0) at (3,0)  {};
			\node[t] (p0) at (4,0)  {};
			\node[t] (p0) at (5,0)  {};
			\node[t] (p0) at (6,0)  {};
		 \draw [ultra thin, loosely dashed](-6,0) -- (6,0);
		 \end{scope}
		}

\node at (7,0) {$\leftarrow R_0$};

\node at (4.5,5) {$\leftarrow R_k$};
\node at (6.3,-2.8) {$\leftarrow R_n$};
	\end{tikzpicture}
	\end{center}\caption{The point set $P$.}\label{fig:third}
\end{figure}

\medskip

\noindent {\bf Remark.}  The number of halving lines of $P$ is 
$\Omega\left(ne^{{\sqrt{(\log 2)/11 }\sqrt{\log n}}}/\sqrt{\log n}\right)$,
 while in the construction of Nivasch 
\cite{N08} it is $\Omega\left(ne^{\sqrt{2 \log 2}\sqrt{\log n}}/\sqrt{\log n}\right)$.
\medskip

\noindent {\bf Remark.} The  point set $P$ is 4-dense, it can be proved by
the same calculation as in \cite{EVW97}.  
We omit the details.


\section{Construction in the space}

\noindent {\bf Proof of Theorem 2.}
Suppose first that $d=3$.
Let $m>0$ even.  
By Lemma 2, there is a planar set $S$ with the 
following properties. 

\begin{enumerate}
 
\item[(a)] The number of points $|S|=m$,

\item[(b)] the number of halving edges of  $S$ is  
$me^{\Omega\left({\sqrt{\log m}}\right)}$,

\item[(c)] for any two points of $S$, the difference of their 
$x$-coordinates is at most $1$ and at least 
$\Omega(m^{-1})$.

\end{enumerate}

Let $\varepsilon>0$ be a very small number. We set its value later. 
By an application of a suitable 
affine transformation (flattening) we can 
assume in addition that

\begin{enumerate}

\item[(d)] each point of $S$ has $y$-coordinate $|y|< \varepsilon^3$.

\end{enumerate}

First, we define two planar sets of points, {\em block $A$} and {\em block $B$}.
Block $A$ contains two parts, the {\em important part} and the 
{\em unimportant part}. Both parts contain $m$ points.
The important part is a translated copy of $S$ 
such that all points have $x$-coordinate 
$1\le x\le 2$ and $y$-coordinate $|y|<\varepsilon^3$.
Then all halving lines of $S$ are very close to the $x$-axis.
The unimportant part contains the points $p_1, \ldots, p_{m/2}$ 
and $q_1, \ldots, q_{m/2}$, where
$p_i=(-2+i/m, \varepsilon^2)$, $q_i=(-1.5+i/m, -\varepsilon^2)$. 
If $\varepsilon$ is small enough, then all halving 
lines of the important part will have 
$p_1, \ldots, p_{m/2}$ on one side and 
$q_1, \ldots, q_{m/2}$ on the other. 
Therefore, they are also halving lines of the whole block $A$.
The origin $O=(0,0)$ is called the {\em center},  
the $x$-axis is called the {\em axis} and
the $xy$ plane is called the {\em plane} of block $A$.

Block $B$ contains points $r_1, \ldots, r_m$ and $s_1, \ldots, s_m$ where
$r_i=(-2+i/m, \varepsilon)$, $s_i=(1+i/m, -\varepsilon)$. 
Again  $O=(0,0)$ is called the {\em center}, the $x$-axis is called the {\em axis} and
the $xy$ plane is called the {\em plane}  of block $B$. Note that $B$ is
symmetric about the origin.

Now take a maximal symmetric (about the origin) packing of discs, 
of spherical radii $1/m$, on the unit sphere, whose center is the origin.
Let $c_1, \ldots, c_k$, $c'_1, \ldots, c'_k$ be their centers, 
$c'_i$ is the reflection of $c_i$. 
Since the packing was maximal, the discs of spherical radii $2/m$ around  
$c_1, \ldots, c_k$, $c'_1, \ldots, c'_k$ cover
the sphere. Therefore, $k=\Theta(m^2)$. 
On the other hand, any two centers are at 
distance at least $1/m$ (actually, almost $2/m$). 
Suppose for simplicity that $k$ is even. 
Perturb the points $c_i$ so that 
no three of them determine a plane through the origin 
and none of them is on the $z$-axis.
Let $\ell_1, \ldots, \ell_{k/2}$,  
be the lines through the origin and 
$c_1, \ldots, c_{k/2}$, respectively, and  
let $\ell'_1, \ldots, \ell'_{k/2}$ 
be the lines through the origin and 
$c_{k/2+1}, \ldots, c_k$, respectively.  
For each line $\ell_i$, $1\le i\le k/2$, 
take a block $A$ so that its center is the origin and its axis is $\ell_i$. 
Its plane can be arbitrary through $\ell_i$ that does not go through any other
$c_j$, $j\neq i$.
For each line $\ell'_i$,  $1\le i\le k/2$, 
take a block $B$ so that its center is the origin and its axis is $\ell'_i$.
Again, its plane can be arbitrary through $\ell'_i$ that does not go through
any other $c_j$.  Choose now the parameter $\varepsilon>0$ to be a very small number. 
Finally slightly perturb the points so that they are in general position.
We obtain the point set $P$. 
Now we have  $n=\Theta(m^3)$ points. The maximum distance is at most $4$, 
the minimum is at least $1/m=\Theta({n^{-1/3}})$, 
so $P$ is $\gamma$-dense for some $\gamma$.

Let {\bf A} be a block of type $A$, 
and {\bf B} a block of type $B$. 
Take two points, $u$, $v$, of the important part of block {\bf A}
that determine a halving line and a point $w$ of block {\bf B}.
The plane $\Pi$ determined by these three points is almost a halving plane of $P$.
It halves any other block, since it goes almost through the origin, 
it also halves block {\bf A}, by the choice of $u$ and $v$, 
and since {\bf B} is almost symmetric about the origin, 
the plane has one more points of {\bf B} on one side than on the other. See Figure~\ref{fig:blockB}.
So, $\Pi$ has one more or one less points of $P$ above than below it. 
An easy calculation shows that we have 
$k^2m^2e^{\Omega\left({\sqrt{\log n}}\right)}=n^2e^{\Omega\left({\sqrt{\log
      n}}\right)}$ 
such planes. We can assume without loss
of generality that at least half of them have one {\em less} points of $P$ above than below. 
Add the point $(0, 0, 2)$ to $P$. It is still dense, and now it has  
$n^2e^{\Omega\left({\sqrt{\log n}}\right)}$ halving planes. 

This proves the result in three dimensions, 
if the number of points  $n=\Theta(m^3)$ for some even $m$
and $n$ is odd. In order to extend the result to all odd 
$n$, we have to add some extra points to the
construction, in pairs, symmetric about the origin. 
We can do it so that it 
does not decrease the number of halving planes
and the resulting set is still dense.

\smallskip

Suppose now that $d>3$. The construction is analogous, we only sketch it.
We use the same blocks, $A$ and $B$ as before.
We take a maximal symmetric (about the origin) packing of $d-1$-dimensional discs, 
of spherical radii $1/m$, on the unit sphere, whose center is the origin.

Let $c_1, \ldots, c_k$, $c'_1, \ldots, c'_k$ be their centers, 
$c'_i$ is the reflection of $c_i$. 
Since the packing was maximal, 
%
%
$k=\Theta(m^{d-1})$. Suppose that $k$ is even. 
%
%
%
Let $\ell_1, \ldots, \ell_{k/2}$, $\ell'_1, \ldots, \ell'_{k/2}$ 
be the lines through the origin and 
$c_1, \ldots, c_k$, respectively.  
For each line $\ell_i$, $1\le i\le k/2$, 
take a block $A$ so that its center is the origin and its axis is $\ell_i$. 
For each line $\ell'_i$,  $1\le i\le k/2$, 
take a block $B$ so that its center is the origin and its axis is $\ell'_i$.
Finally slightly perturb the points so that they are in general position.
We obtain the point set $P$. 
Now we have  $n=\Theta(m^d)$ points. The maximum distance is at most $4$, 
the minimum is at least $1/m=\Theta({n^{-1/d}})$, 
so $P$ is $\gamma$-dense for some $\gamma$.

Let {\bf A} be a block of type $A$, 
and let ${\mbox{\bf B}}_1,  {\mbox{\bf B}}_2, \ldots, {\mbox{\bf B}}_{d-2}$
be different blocks, each  
of type $B$. 
Take two points, $u$, $v$, of the important part of block {\bf A}
that determine a halving line and for $1\le i\le d-2$, 
let
$w_i\in {\mbox{\bf B}}_{i}$. 
The hyperplane $\Pi$ determined by these $d$ points is almost a halving plane of $P$.
It halves any other block, since it goes almost through the origin, 
it also halves block {\bf A}, by the choice of $u$ and $v$, 
and since ${\mbox{\bf B}}_{i}$
is almost symmetric about the origin, 
$\Pi$
has one more points of ${\mbox{\bf B}}_{i}$
on one side than on the other. 
An easy calculation shows that we have 
$n^{d-1}e^{\Omega\left({\sqrt{\log n}}\right)}$ such planes.

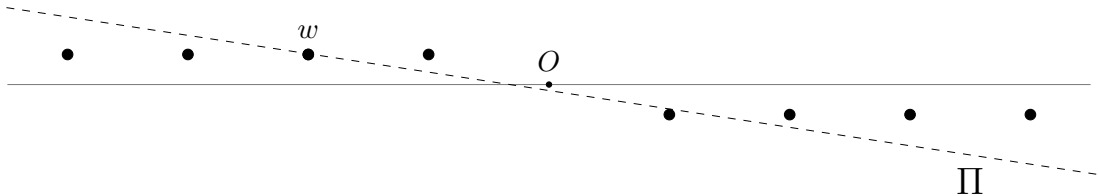
\begin{figure}
 \begin{center}
 \begin{tikzpicture}[] 
	 \begin{scope}[scale=0.8]
 	 \tikzstyle{t}=[draw,circle,fill=black,minimum size=4,inner sep=0]
 	 \draw [gray](-9,0) -- (9,0);
 	 \node[draw,circle,fill=black,minimum size=2,inner sep=0] (p0) at (0,0) [label=above:$O$] {};
 \newcommand{\deltanullb}{0.2}

 	\forloop{ct}{0}{\value{ct} < 4}%
 	{%
 	\begin{scope}[shift={(2*\arabic{ct},0.5)}]
 	\node[t] (p0) at (-8,0) [] {};
 	 \end{scope}
 	}
	
 	\forloop{ct}{0}{\value{ct} < 4}%
 	{%
 	\begin{scope}[shift={(2*\arabic{ct},-0.5)}]
 	\node[t] (p0) at (2,0) [] {};
 	 \end{scope}
 	}
	
 	\begin{scope}[shift={(-6,0.5)}]
 	\node[t] (p0) at (2,0) [label=above:$w$] {};
 	 \end{scope}
	 
	 \node at (7,-1.6) {\Large $\Pi$};

\begin{scope}[shift={(0,-0.1)}]
	\begin{scope}[rotate=171.3]
 	\begin{scope}[shift={(\deltanullb,0)}]
	 \draw [ultra thin, dashed](-9 - 2* \deltanullb,0) -- (9 ,0);
 	 \end{scope}
	 \end{scope}
\end{scope}
	
		  \end{scope}
 \end{tikzpicture}
 \end{center}\caption{Block {\bf B} of type $B$ and plane $\Pi$.}\label{fig:blockB}
\end{figure}

Let $\mbox{\em diff}(\Pi)$ be the number of points of $P$ {\em above} $\Pi$ minus
the number of points {\em below} $\Pi$. (``Above'' and ``below'' are defined
with respect to the $d$-th coordinate.) 
It follows, that  
$|\mbox{\em diff}(\Pi)|\le d-1$ and $\mbox{\em diff}(\Pi)-d$ is even. 
So, for some 
$x$, where $|x|\le d-1$ and $x-d$ is even, 
there are still 
$n^{d-1}e^{\Omega\left({\sqrt{\log n}}\right)}$ such planes $\Pi$ with $\mbox{\em diff}(\Pi)=x$.
We can assume without loss
of generality that $x\le 0$. 
Add 
the points $(0, 0, \ldots, 0, 2+i/n)$, $1\le i\le x$ to $P$. 
The resulting set is still dense and now it has  
$n^2e^{\Omega\left({\sqrt{\log n}}\right)}$ halving planes.

\bigskip

\noindent {\bf Acknowledgement.} We are very grateful to the anonymous 
referees for 
their extremely thorough work. Their remarks enormously 
improved the presentation of the paper.

\end{document}